\documentclass[smallextended,envcountsect]{svjour3} 

\usepackage{amsmath}
\usepackage{amssymb}
\usepackage{cite}
\usepackage{dsfont}
\usepackage{hyperref}
\usepackage{enumitem}
\usepackage{lmodern,color}

\usepackage{todonotes}

\newcommand{\A}{\mathcal{A}}
\renewcommand{\H}{\mathcal{H}}

\makeatletter
\def\namedlabel#1#2{\begingroup
    #2%
    \def\@currentlabel{#2}%
    \phantomsection\label{#1}\endgroup}
\makeatother

\newenvironment{claimproof}[2]{\par\noindent\textit{Proof of Claim} #1:\space#2}{\hfill $\square$}

\smartqed 
\usepackage{graphicx}
\journalname{}

\begin{document}

\title{Moreau-Yosida Regularization of Degenerate State-Dependent Sweeping Processes}
\titlerunning{Moreau-Yosida Regularization of Degenerate State-dependent Sweeping Processes}
\subtitle{}

\author{Diana Narv\'aez \and  Emilio Vilches}

\institute{Diana Narv\'aez \at
             Departamento de Ingenier\'ia Matem\'atica\\ 
           Universidad de Chile\\
           Santiago, Chile\\
           \email{dnarvaez@dim.uchile.cl}
           \and
					 Emilio Vilches, Corresponding author \at
           Instituto de Ciencias de la Ingenier\'ia\\ 
           Universidad de O'Higgins\\
           Rancagua, Chile\\
           \email{emilio.vilches@uoh.cl}
}
\date{Received: date / Accepted: date}

\maketitle

\begin{abstract}
In this paper, we introduce and study  degenerate state-dependent sweeping processes with nonregular moving sets (subsmooth and positively $\alpha$-far). Based on the Moreau-Yosida regularization, we prove the existence of solutions under the Lipschitzianity of the moving sets with respect to the truncated Hausdorff distance. 
\end{abstract}



\section{Introduction}

The sweeping process is a first-order  dynamical system involving normal cones to moving sets. It was introduced and deeply studied by J. J. Moreau as a model for an elastoplastic mechanical system (see, e.g.,  \cite{MO1,MO2,MO4,Moreau1999}). Since then, and due to its versatility as a mathematical model, it has been used in different applications such as electrical circuits  \cite{Acary-Bon-Bro-2011}, contact mechanics \cite{Brogliato-M}, crowd motion \cite{Maury-Venel} , hysteresis phenomena \cite{Krejci-1997}, soft crawlers \cite{Gidoni1998}, among others.  Furthermore, the sweeping process has motivated the use of differential inclusions with normal cones, which has allowed the study of new variants of the sweeping process. Among them, we can mention the state-dependent sweeping process \cite{JV-Moreau}, the second-order sweeping process \cite{Adly2016,MM2019}, the implicit sweeping process \cite{JV-Implicit,JZ-2020}, the generalized sweeping process  \cite{JV-alpha,JV-Galerkin2017}, and the degenerate sweeping process \cite{Kunze-MM1997}.  All of them are motivated by practical applications, and where there are still several open problems both from the practical and mathematical point of view. We refer to \cite{Brogliato-Tanwani} for a recent survey on the subject.

In this article, we are concerned with the study of state-dependent sweeping processes and  degenerate sweeping processes.  

On the one hand, the state-dependent sweeping process corresponds to the case where the moving set also depends on the state. This differential inclusion has been motivated by quasi-variational inequalities arising, e.g., in  quasistatic evolution problems with friction, the evolution of sandpiles and micromechanical damage models for iron materials, among others (see \cite{Kunze-MM2000} and the references therein).  We refer to \cite{JV-Moreau} for a review on the subject.

On the other hand, the degenerate sweeping process corresponds to the case where a linear/nonlinear operator is added ``inside'' the sweeping process.  This dynamics was proposed by Kunze and Monteiro-Marques  as a model for quasistatic elastoplasticity (see \cite{Kunze-MM1997}). Since then,  the degenerate sweeping process has been studied by several authors in the context of convex and prox-regular sets (see \cite{Kunze-MM1997,Kunze1998,Kunze1999,Adly2020,Haddad-Sene2019}). 

Regarding the moving sets, both the state-dependent sweeping process and the degenerate sweeping process have been studied by considering continuity properties with respect to the Hausdorff distance. However, it has been clear that the Hausdorff distance is quite restrictive and thus limits the spectrum of possible applications (see, e.g., \cite{Thibault2016,NacryT2019}). We provide an example of moving sets where the Lipschitzianity of the moving sets with respect to the Hausdorff distance is not satisfied, but the weaker notion of Lipschitzianity with respect to the truncated Hausdorff distance holds. 

In this paper, we consider both processes with nonregular moving sets satisfying a notion of Lipschitzianity with respect to the truncated Hausdorff distance, which generalizes all the known results in the literature. 
Another novelty of this work is that we introduced and study the degenerate state-dependent sweeping process with nonregular sets (subsmooth and positively $\alpha$-far), which is also new in the literature.

This paper aims to 1) to unify the state-dependent sweeping process with the degenerate sweeping process. 2) to relax the Lipschitzianity of the moving sets with respect to the Hausdorff distance by the weaker notion of  Lipschitzianity with respect to the truncated Hausdorff distance. 3) to study the  degenerate state-dependent sweeping process with nonregular moving sets (subsmooth and positively $\alpha$-far).  

Our results are obtained through the Moreau-Yosida regularization technique. The Moreau-Yosida regularization technique consists of approaching a given differential inclusion by a penalized one, depending on a parameter,  whose existence is easier to obtain (for example, by using the classical Cauchy-Lipschitz Theorem), and then to pass to the limit as the parameter goes to zero. This technique has been used several time to deal with sweeping processes (see \cite{MO1,Marques87,Marques1993,Kunze1996,Thibault2008,Mazade2013,Sene2014} for more details). Recently, it was extended by Jourani and Vilches \cite{JV-Moreau} to deal with state-dependent sweeping processes with Lipschitz nonregular moving sets with respect to the Hausdorff distance.   We adapt and extend the techniques from \cite{JV-Moreau} to deal with the degenerate state-dependent sweeping process driven by Lipschitz nonregular moving sets with respect to the truncated Hausdorff distance. 

The paper is organized as follows. After some preliminaries,  in Section \ref{hipo-sol}, we list the hypotheses used throughout the paper.  In Section \ref{preparatory},  we establish some technical results used to prove the main result of the paper. Next, in Section \ref{Existence}, we present  the main result of the paper  (Theorem \ref{Lipschitz-state}), namely, the convergence (up to a subsequence) of the Moreau-Yosida regularization for the degenerate state-dependent sweeping process under the Lipschitzianity of the moving sets with respect to the truncated Hausdorff distance.  The paper ends with final comments and conclusions.

\section{Preliminaries}

In this paper, $(\H,\langle \cdot,\cdot\rangle)$ stands for a separable Hilbert with norm  $\Vert \cdot \Vert$. We denote by $\mathbb{B}(x,\rho)$ the closed ball centered at $x$ with radius $\rho$.  The closed unit ball is denoted by $\mathbb{B}$. We write $\H_w$ for the space $\H$ endowed with the weak topology and $x_n \rightharpoonup x$ stands for the weak convergence in $\H$ of a sequence $(x_n)_n$ to $x$.

A vector $h\in \H$ belongs to the Clarke tangent cone $T_C(S;x)$, when for every sequence $(x_n)_n$ in $S$ converging to $x$ and every sequence of positive numbers $(t_n)_n$ converging to $0$, there exists some sequence $(h_n)_n$ in $\H$ converging to $h$ such that $x_n+t_nh_n\in S$ for all $n\in \mathbb{N}$. This cone is closed and convex and its negative polar is the Clarke normal cone to $S$ at $x\in S$, that is,
$N\left(S;x\right)=\left\{v\in \H\colon \left\langle v,h\right\rangle \leq 0 \textrm{ for all }  h\in T_C(S;x)\right\}$.
As usual, $N(S;x)=\emptyset$ if $x\notin S$. Through the Clarke normal cone, we define the Clarke subdifferential of a function $f\colon \H\to \mathbb{R}\cup\{+\infty\}$ as
\begin{equation*}
\partial f(x):=\left\{v\in \H\colon (v,-1)\in N\left(\operatorname{epi}f,(x,f(x))\right)\right\},
\end{equation*}
where $\operatorname{epi}f:=\left\{(y,r)\in \H\times \mathbb{R}\colon f(y)\leq r\right\}$ is the epigraph of $f$. When the function $f$ is finite and locally Lipschitzian around $x$, the Clarke subdifferential is characterized (see \cite[Proposition~2.1.5]{Clarke1998}) in the following way
\begin{equation*}
\partial f(x)=\left\{v\in \H\colon \left\langle v,h\right\rangle \leq f^{\circ}(x;h) \textrm{ for all } h\in \H\right\},
\end{equation*}
where
\begin{equation*}
f^{\circ}(x;h):=\limsup_{(t,y)\to (0^+,x)}t^{-1}\left[f(y+th)-f(y)\right],
\end{equation*}
is the \emph{generalized directional derivative} of the locally Lipschitzian function $f$ at $x$ in the direction $h\in \H$.  The function $f^{\circ}(x;\cdot)$ is in fact the support of $\partial f(x)$, that is, $f^{\circ}(x;h)=\sup\{\langle v,h\rangle\colon v\in \partial f(x)\}$. This characterization easily yields that the Clarke subdifferential of any locally Lipschitzian function has the important property of \emph{  upper semicontinuity from $\H$ into $\H_w$}.

For $S\subset \H$, the distance function is defined by $d_{S}(x):=\inf_{y\in S}\Vert x-y\Vert$ for $x\in \H$. We denote $\operatorname{Proj}_{S}(x)$ as the set (possibly empty) of points which attain this infimum.  The equality (see \cite[Proposition~2.5.4]{Clarke1998})
\begin{equation}\label{cono-distance}
\begin{aligned}
N\left(S;x\right)&=\operatorname{cl}^{w}\left({\mathbb{R}_+\partial d_S(x)}\right) & \textrm{ for } x\in S,
\end{aligned}
\end{equation}
gives an expression of the Clarke normal cone in terms of the distance function. As usual, it will be convenient to write $\partial d(x,S)$ in place of $\partial d\left(\cdot,S\right)(x)$.
\begin{remark}
In the present paper,  we will calculate the Clarke subdifferential of the distance function to a moving set. By doing so, the subdifferential will always be calculated with respect to the variable involved in the distance function by assuming that the set is fixed. More explicitly, $\partial d_{C(t,y)}(x)$ means the subdifferential of the function $d_{C(t,y)}(\cdot{})$ (here $C(t,y)$ is fixed) is calculated at the point $x$, i.e., $\partial \left(d_{C(t,y)}(\cdot)\right)(x)$. In the same way, $\partial d_{C(t,x)}(x)$ means the subdifferential of the function $d_{C(t,x)}(\cdot)$ (here $C(t,x)$ is a fixed set) is calculated at the point $x$, i.e., $\partial \left(d_{C(t,x)}(\cdot)\right)(x)$.
\end{remark}

Let $f\colon \H\to \mathbb{R}\cup\{+\infty\}$ be a proper and lsc function and $x\in \operatorname{dom}f$. An element $\zeta$ belongs to the proximal subdifferential (see \cite[Chapter~1]{Clarke1998}) $\partial_P f(x)$ of $f$ at $x$  if there exist two positive numbers $\sigma$ and $\eta$ such that
\begin{equation*}
\begin{aligned}
f(y)&\geq f(x)+\left\langle \zeta,y-x\right\rangle -\sigma\Vert y-x\Vert^2 & \forall y\in B(x;\eta).
\end{aligned}
\end{equation*}
For a closed set $S\subset \H$, the \emph{proximal normal cone} of $S$ at $x\in S$ is defined as $N^P(S;x)=\partial_P I_S(x)$, where $I_S$ is the indicator function of the set $S$.

\noindent Now, we recall the definition of the class of positively $\alpha$-far sets, introduced in \cite{Haddad2009} and widely studied in \cite{JV-alpha}.
\begin{definition} Let $\alpha\in ]0,1]$ and $\rho\in ]0,+\infty]$. Let $S$ be a nonempty closed subset of $\H$ with $S\neq \H$. We say that the Clarke subdifferential of the \,distance function $d(\cdot,S)$ keeps the origin $\alpha$-far-off on the open $\rho$-tube around $S$, $U_{\rho}(S):=\{x\in \H\colon 0<d(x,S)<\rho\}$, provided
\begin{equation}\label{13}
0<\alpha\leq \inf_{x\in U_{\rho}(S)}d(0,\partial d(\cdot,S)(x)).
\end{equation}
Moreover, if $E$ is a given nonempty set, we say that the family $(S(t))_{t\in E}$ is positively $\alpha$-far if every $S(t)$ satisfies (\ref{13}) with the same $\alpha\in ]0,1]$ and  $\rho>0$.
\end{definition}
Several characterizations of this notion and examples were given in \cite{JV-alpha}. We refer to \cite{JV-alpha} for more details. 

The notion of positively $\alpha$-far sets includes strictly the notion of uniformly subsmooth sets (see Proposition \ref{alpha-far}) and the notion of uniformly prox-regular sets (see \cite{JV-alpha}).
\begin{definition}[\cite{Poliquin2000}]
  For a fixed $r>0$, the set $S$ is said to be $r$-uniformly prox-regular if for any $x\in S$ and $\zeta\in N_S^P(x)\cap \mathbb{B}$  one has $x=\operatorname{proj}_S(x+r\zeta)$.
\end{definition}
 It is known that $S$ is $r$-uniformly prox-regular if and only if every nonzero proximal vector $\zeta\in N_S^P(x)$ to $S$ at any point $x\in S$ can be realized by an $r$-ball, that is, $S\cap B\left(x+r\frac{\zeta}{\Vert \zeta\Vert}\right)=\emptyset$, which is equivalent to
\begin{equation*}
\begin{aligned}
\left\langle \zeta,y-x\right\rangle&\leq \frac{\Vert \zeta\Vert}{2r}\Vert y-x\Vert^2 & \textrm{ for all } y\in S.
\end{aligned}
\end{equation*}

\begin{definition}Let $S$ be a closed subset of $\H$. We say that $S$ is \emph{uniformly subsmooth}, if for every $\varepsilon>0$ there exists $\delta>0$, such that
\begin{equation}\label{14}
  \left\langle x_1^*-x_2^*,x_1-x_2\right\rangle \geq -\varepsilon \Vert x_1-x_2\Vert
\end{equation}
holds for all $x_1,x_2\in S$ satisfying $\Vert x_1-x_2\Vert <\delta$ and all $x_i^*\in N\left(S;x_i\right)\cap \mathbb{B}$ for $i=1,2$. Also, if $E$ is a given nonempty set, we say that the family $\left(S(t)\right)_{t\in E}$ is equi-uniformly subsmooth, if for every $\varepsilon>0$, there exists $\delta>0$ such that (\ref{14}) holds for each $t\in E$ and all $x_1,x_2\in S(t)$ satisfying $\Vert x_1-x_2\Vert <\delta$ and all $x_i^*\in N\left(S(t);x_i\right)\cap \mathbb{B}$ for $i=1,2$.
\end{definition}
\begin{proposition}[\cite{JV-alpha}]\label{alpha-far}
 Assume that $S$ is uniformly subsmooth. Then, for all $\varepsilon\in ]0,1[$ there exists $\rho\in ]0,+\infty[$ such that 
\begin{equation*}
\sqrt{1-\varepsilon}\leq \inf_{y\in U_{\rho}(S)}d(0,\partial d(y,S)).
\end{equation*}
\end{proposition}

\begin{remark}The class of positively $\alpha$-far sets contains strictly that of uniformly subsmooth sets. To see this, consider $S=\{(x,y)\in \mathbb{R}^2\colon y\geq  -|x|\}$. Then, (see \cite{Haddad2009}), $S$ satisfies relation (\ref{13}) with  $\alpha=\frac{\sqrt{2}}{2}$ on $\H\setminus S$, but we easily see that $S$ is not uniformly subsmooth.
\end{remark}

Let $A$ be a bounded subset of $\H$. We define the Kuratowski measure of non-compactness of $A$, $\alpha(A)$, as
\begin{equation*}
\alpha(A)=\inf\{d>0\colon A \textrm{ admits a finite cover by sets of diameter }\leq d\},
\end{equation*}
and the Hausdorff measure of non-compactness of $A$, $\beta(A)$, as
\begin{equation*}
\beta(A)=\inf\{r>0\colon A \textrm{ can be covered by finitely many balls of radius } r\}.
\end{equation*}
The following result gives the main properties of the Kuratowski  and Hausdorff measure of non-compactness (see \cite[Proposition~9.1 from  Section~9.2]{Deimling1992}).
\begin{proposition}\label{Kura}
  Let $\H$ be a Hilbert space and $B,B_1,B_2$ be bounded subsets of $\H$. Let $\gamma$ be the Kuratowski or the Hausdorff measure of non-compactness. Then,
  \begin{enumerate}[label=(\roman{*})]
    \item $\gamma(B)=0$ if and only if $\operatorname{cl}(B)$ is compact.
    \item $\gamma(\lambda B)=|\lambda|\gamma(B)$ for every $\lambda\in \mathbb{R}$.
    \item  $\gamma(B_1+B_2)\leq \gamma(B_1)+\gamma(B_2)$.
    \item $B_1\subset B_2$ implies $\gamma(B_1)\leq \gamma(B_2)$.
    \item $\gamma(\operatorname{conv}B)=\gamma(B)$.
    \item $\gamma(\operatorname{cl}(B))=\gamma(B)$.
    \item  If $\A \colon \H \to \H$ is a Lipschitz map of constant $M\geq 0$, then $$\gamma(\A(B))\leq M\gamma(B).$$
  \end{enumerate}
\end{proposition}

The following result (see \cite[Theorem~2]{Bothe1998}) is used to prove the existence of the Moreau-Yosida regularization scheme.
\begin{theorem}\label{bothe-2}
 Let $I=[0,T]$ for some $T>0$ and $F\colon I\times \H \rightrightarrows \H$ be with a set-valued map with nonempty closed and convex values satisfying:
\begin{enumerate}[label=(\roman{*})]
\item for every $x\in \H$, $F(\cdot,x)$ is measurable.
\item for every $t\in I$, $F(t,\cdot)$ is upper semicontinuous from $\H$ into $\H_w$.
\item for a.e. $t\in I$ and $A\subset \H$ bounded, $\gamma(F(t,A))\leq k(t)\gamma(A)$, 
for some $k\in L^1(0,T)$ with $k(t)<+\infty$ for all $t\in I$, where $\gamma=\alpha$ or $\gamma=\beta$ is either the Kuratowski or the Hausdorff measure of non-compactness.
\end{enumerate}
Then, the differential inclusion
\begin{equation*}
\left\{
\begin{aligned}
\dot{x}(t)&\in F(t,x(t)) & \textrm{ a.e. } t\in I,\\
x(0)&=x_0,
\end{aligned}
\right.
\end{equation*}
has at least one solution $x\in \operatorname{AC}\left(I;\H\right)$.
\end{theorem}
The following result provides a compactness criterion for absolutely continuous functions (see \cite[Lemma~2.2]{JV-Moreau}).
\begin{lemma}\label{compactness}
 Let $(x_n)_n$ be a sequence of absolutely continuous functions from $I$ into $\H$ with $x_n(0)=x_0^n$.  Assume that for all $n\in \mathbb{N}$
\begin{equation}\label{acotamiento}
\begin{aligned}
\Vert \dot{x}_n(t)\Vert &\leq \psi(t) & \textrm{ a.e } t\in I,
\end{aligned}
\end{equation}
where $\psi\in L^1(0,T)$ and that $x_0^n \to x_0$ as $n\to \infty$. Then, there exists a subsequence $(x_{n_k})_k$ of $(x_n)_n$ and $x\in \operatorname{AC}\left(I;\H\right)$ such that
\begin{enumerate}[label=(\roman{*})]
\item\label{comp-i}  $x_{n_k}(t)\rightharpoonup x(t)$ in $\H$ as $k\to +\infty$ for all $t\in I$.
\item\label{comp-ii} $x_{n_k}\rightharpoonup x$ in $L^1\left(I;\H\right)$ as $k\to +\infty$.
\item\label{comp-iii} $\dot{x}_{n_k}\rightharpoonup \dot{x}$ in $L^1\left(I;\H\right)$ as $k\to +\infty$.
\item\label{comp-iv} $\Vert \dot{x}(t)\Vert \leq \psi(t)$ a.e. $t\in I$.
\end{enumerate}
\end{lemma}

\noindent \textbf{Truncated Hausdorff Distance}

Let $r\in ]0,+\infty]$ be a given extended real and let $A$ and $B$ be nonempty subsets of $\H$.  We define the $r$-\emph{truncated excess} of $A$ over $B$ as the extended real
$$
\operatorname{exp}_{r}(A,B):=\sup_{x\in A\cap r\mathbb{B}}d(x,B).
$$
It is clear that under the usual convention $r\mathbb{B}=\H$ for $r=+\infty$, the  $r$-truncated excess of $A$ over $B$ is the usual \emph{excess} of $A$ over $B$, i.e., 
$$
\operatorname{exc}_{\infty}(A,B)=\sup_{x\in A}d(x,B)=:\operatorname{exc}(A,B).
$$
In this paper, our results are based on the following basic notions of distance between sets. 
\begin{itemize}
\item $$
{\operatorname{haus}}_{r}(A,B):=\max\{\operatorname{exc}_{r}(A,B),\operatorname{exc}_r(B,A)\}
$$
\item $$
\widehat{\operatorname{haus}}_{r}(A,B):=\sup_{z\in r\mathbb{B}}\vert d(z,A)-d(z,B)\vert. 
$$
\end{itemize}
The quantity ${\operatorname{haus}}_{r}(A,B)$ receives the name of $r$-\emph{truncated Hausdorff distance}. It is important to emphasize that  $$\operatorname{haus}_r(A,B)\leq \widehat{\operatorname{haus}}_{r}(A,B).$$
Moreover, for any extended real $r^{\prime}\geq 2r +\max\{ d_A(0),d_B(0)\}$, the following inequality holds
$$
\widehat{\operatorname{haus}}_{r}(A,B)\leq {\operatorname{haus}}_{r^{\prime}}(A,B).
$$
As a consequence of the last two inequalities, we obtain the following result.
\begin{proposition} Let $C\colon I\rightrightarrows \H$ be a set-valued map with nonempty and closed sets. Then, the following assertions are equivalent. 
\begin{enumerate}
\item[(a)]  for all $r\geq 0$, there exists $\kappa_r\geq 0$ such that
$$
\widehat{\operatorname{haus}}_{r}(C(t),C(s))\leq \kappa_r \vert t-s\vert \textrm{ for all } s,t\in I.
$$
\item[(b)] for all $r\geq 0$, there exists $\kappa_r\geq 0$ such that
$$
{\operatorname{haus}}_{r}(C(t),C(s))\leq \kappa_r \vert t-s\vert \textrm{ for all } s,t\in I.
$$
\end{enumerate}
\end{proposition}
Therefore, we use $(a)$ and $(b)$ indistinctly.
\begin{definition}
We say that a set-valued map $C\colon I\rightrightarrows \H$ is 
\begin{enumerate}
\item[(i)] \emph{Lipschitz with respect to the Hausdorff distance} if there exists $\kappa \geq 0$ such that
$$
\operatorname{haus}(C(t),C(s)):=\sup_{z\in \H}\vert d(z,A)-d(z,B)\vert \leq  \kappa \vert t-s\vert \textrm{ for all } s,t\in I.
$$

\item[(ii)] \emph{Lipschitz with respect to the truncated Hausdorff distance} if for all $r\geq 0$, there exists $\kappa_r\geq 0$ such that
$$
\widehat{\operatorname{haus}}_{r}(C(t),C(s))\leq \kappa_r \vert t-s\vert \textrm{ for all } s,t\in I.
$$
\end{enumerate}
\end{definition}
It is clear that (i) implies (ii), but the reverse implication is not true. Moreover, as discussed in \cite{NacryT2019,Tolstonogov2014},  the Hausdorff distance can be too restrictive for practical applications. Indeed, let us consider the half-space moving set $C\colon I\times \H \rightrightarrows \H$ defined by
$$
C(t,x):=\{z\in \H \colon \langle \zeta(t),z\rangle - \beta(t,x) \leq 0\} \textrm{ for all } (t,x)\in I\times \H,
$$
where $\zeta\colon I\to \H$ is a $\gamma$-Lipschitz mapping for some real $\gamma\geq 0$ and the map $\beta\colon I\times \H \to \H$ satisfies the following Lipschitz property:
$$
\vert \beta(s,x)-\beta(t,y)\vert \leq \gamma\vert t-s\vert +L\Vert x-y\Vert \textrm{ for all } s,t \in I \textrm{ and } x,y\in \H.
$$
Assume that the following normalization condition holds true: $\Vert \zeta(t)\Vert=1$ for all $t\in I$. According to \cite[Theorem~6.30]{Deutsch2001}, we have 
$$
d_{C(t,x)}(z)=\left( \langle \zeta(t),z\rangle -\beta(t,x)\right)^{+} \textrm{ for all } (t,x)\in I\times \H,
$$
where $r^+:=\max\{0,r\}$ for all $r\in \mathbb{R}$. Hence, we can check that for each $r>0$ and any $s,t\in I$ with $s\neq t$ and $x,y\in \H$,
$$
\widehat{\operatorname{haus}}_{r}(C(s,x),C(t,y))=\sup_{z\in r \mathbb{B}}\vert d_{C(s,x)}(z)-d_{C(t,y)}(z)\vert \leq \gamma \vert t-s\vert \, (r+1)+L\Vert x-y\Vert.
$$
However,  $\operatorname{Haus}(C(t,x),C(s,y))=+\infty$, which shows that $C$ is Lipschitz with respect to the truncated Hausdorff distance but not Lipschitz with respect to Hausdorff distance.

\section{Main Hypotheses}\label{hipo-sol}

In this section,  we list the main hypotheses used throughout the paper.

\noindent \paragraph{ Hypotheses on the  map $\A\colon \H\to \H$:} $\A$ is a (possibly) nonlinear operator satisfying the following conditions:
\begin{enumerate}
\item[\namedlabel{HA1}{$(\mathcal{H}_{\A}^1)$}] There exists $m>0$ such that 
$$
\langle \A(x)-\A(y),x-y\rangle \geq m\Vert x-y\Vert^2 \textrm{ for all } x,y\in \H.
$$

\item[\namedlabel{HA2}{$(\mathcal{H}_{\A}^2)$}] There exists $M>0$ such that 
$$
\Vert \A(x)-\A(y)\Vert \leq M \Vert x-y\Vert  \textrm{ for all } x,y\in \H.
$$

\end{enumerate}

\paragraph{ Hypotheses on $C\colon I\times \H\rightrightarrows \H$:} $C$ is a set-valued map with closed and nonempty values. Moreover, the following condition will be used in the paper.

\begin{enumerate}

\item[\namedlabel{Hxlip}{$(\mathcal{H}_1)$}] For all $r\geq 0$, there exist  $\kappa_r \geq 0$ such that for $s,t\in I$ and $x,y\in \H$
\begin{equation*}
\sup_{z\in r\mathbb{B}}|d(z,C(t,x))-d(z,C(s,y))|\leq \kappa_r|t-s|+L\Vert x-y\Vert,
\end{equation*}
where  $L\in [0,m[$ is independent of $r$ and $m$ is the constant given by \ref{HA1}.

\item[\namedlabel{Hxalpha}{$(\mathcal{H}_2)$}] There exist constants $\alpha\in ]0,1]$ and $\rho\in ]0,+\infty]$ such that for every $y\in \H$
\begin{equation*}
\begin{aligned}
  0<\alpha&\leq \inf_{x\in U_{\rho}\left(C(t,y)\right)}d\left(0,\partial d(\cdot,C(t,y))(x)\right) & \textrm{ a.e. } t\in I,
  \end{aligned}
\end{equation*}
where $U_{\rho}\left(C(t,y)\right)=\left\{ x\in \H \colon 0<d(x,C(t,y))<\rho \right\}$.

\item[\namedlabel{Hxsubsmooth}{$(\mathcal{H}_3)$}] The family $\{C(t,v)\colon (t,v)\in I\times \H\}$ is equi-uniformly subsmooth.

\item[\namedlabel{Hxcomp}{$(\mathcal{H}_{4})$}] There exists $k\in L^1(0,T)$ such that for every $t\in I$, every $r>0$ and every bounded set $B\subset \H$, 
$$
\gamma\left(C(t,B)\cap r\mathbb{B}\right)\leq k(t)\gamma(\A(B)),
$$
 where $\gamma=\alpha$ or $\gamma=\beta$ is either the Kuratowski or the Hausdorff measure of non-compactness (see Proposition \ref{Kura}), $k(t)<1$ for all $t\in I$ and $\mathcal{A}\colon \H \to \H$ is an operator satisfying \ref{HA1} and \ref{HA2}.
\end{enumerate}

\paragraph{ Hypotheses on  $C\colon I\rightrightarrows \H$:} $C$ is a set-valued map with  closed and nonempty values. Moreover, the following condition will be used in the paper.
\begin{enumerate}

\item[\namedlabel{Hlip}{$(\mathcal{H}_5)$}] For all $r \geq 0$, there exists $\kappa_{r} \geq 0$ such that for all $s,t\in I$ 
\begin{equation*}
\sup_{z\in r \mathbb{B}}\vert d(z,C(t))-d(z,C(s))\vert \leq \kappa_{r}|t-s|.
\end{equation*}

\item[\namedlabel{Halpha}{$(\mathcal{H}_6)$}] There exist two constants $\alpha\in ]0,1]$ and $\rho\in ]0,+\infty]$ such that
\begin{equation*}
\begin{aligned}
  0<\alpha&\leq \inf_{x\in U_{\rho}\left(C(t)\right)}d\left(0,\partial d(x,C(t))\right) & \textrm{ a.e. } t\in I,
  \end{aligned}
\end{equation*}
where $U_{\rho}\left(C(t)\right)=\left\{ x\in \H \colon 0<d(x,C(t))<\rho \right\}$.

\item[\namedlabel{Hcomp}{$(\mathcal{H}_7)$}] For a.e. $t\in I$ the set $C(t)$ is ball-compact, that is, for every $r>0$ the set $C(t)\cap r\mathbb{B}$ is compact in $\H$.

\item[\namedlabel{Hprox}{$(\mathcal{H}_8)$}] For a.e. $t\in I$ the set $C(t)$ is $r$-uniformly prox-regular for some $r>0$.
\end{enumerate}

\begin{remark} 
\begin{enumerate}[label=(\roman{*})]
\item Let $L\in [0,m[$. Under \ref{Hxsubsmooth} for every $\alpha\in ]\sqrt{\frac{L}{m}},1]$ there exists $\rho>0$ such that \ref{Hxalpha} holds. This follows from Proposition \ref{alpha-far}.
\item The condition $L\in [0,m[$ in \ref{Hxlip} is sharp, as it is shown in \cite{Kunze1998}.
\end{enumerate}
\end{remark}

\section{Preparatory Lemmas}\label{preparatory}
In this section, we give some preliminary lemmas that will be used in the following sections. They are related to set-valued maps and properties of the distance function.

Recall that $-d(\cdot, S )$ has a directional  derivative that coincides with the Clarke directional derivative of $-d(\cdot, S )$ whenever $x\notin S$ (see, e.g., \cite{BFG1987}). Thus, we obtain the following lemma.
\begin{lemma}\label{lema0}
  Let $S\subset \H$ be a closed set, $x\notin S$ and $v\in \H$. Then
\begin{equation*}
\lim_{h\downarrow 0} \frac{d(x+hv,S)-d(x,S)}{h}= \min_{y^*\in \partial d(x,S)} \left\langle y^*, v \right \rangle.
\end{equation*}
\end{lemma}

The next lemma characterizes the Clarke subdifferential of the distance function to a moving sets (see \cite[Lemma~4.2]{JV-Moreau}).
\begin{lemma}[\cite{JV-Moreau}]\label{lemma-formula}
 Assume that \ref{Hxcomp} holds. Let $t\in I$, $y\in \H$ and $x\notin C(t,y)$. Then,
\begin{equation*}
\partial d_{C(t,y)}(x)=\frac{x-\operatorname{cl}{\operatorname{co}}\operatorname{Proj}_{C(t,y)}(x)}{d_{C(t,y)}(x)}.
\end{equation*}
\end{lemma}

The following result can be proved in the same way that  \cite[Lemma~4.3]{JV-Moreau}
\begin{lemma}\label{usc}
Let $\A\colon \H \to \H$ be a Lipschitz operator.  If  \ref{Hxlip}, \ref{Hxsubsmooth} and \ref{Hxcomp} hold then, for all $t\in I$, then the set-valued map $x\rightrightarrows \partial d(\cdot,C(t,x))(\A(x))$ is upper semicontinuous from $\H$ into $\H_w$.
\end{lemma}

The following lemma provides some estimations for the distance function to a moving set depending on time and state. 
\begin{lemma}\label{derivadas} Let  $\A\colon \H \to \H$ be a map satisfying \ref{HA1} and \ref{HA2}, $x, y\colon I\to \H$ be two absolutely continuous functions, and $C\colon I\times\ \H \rightrightarrows \H$ be a set-valued map with nonempty closed values satisfying \ref{Hxlip}. Set $z(t):=\A(x(t))$ for all $t\in I$. Then,
  \begin{enumerate}[label=(\roman{*})]
\item The function $t \to d(z(t),C(t,y(t)))$ is absolutely continuous over $I$.
\item For all $t\in \operatorname{int}(I)$, where $\dot{y}(t)$ exists,
\begin{equation*}
\begin{aligned}
&\limsup_{s\downarrow 0} \frac{d_{C(t+s,y(t+s))}(z(t+s))-d_{C(t,y(t))}(z(t))}{s}\\
&\leq \kappa_r+L{\Vert \dot{y}(t)\Vert}+
\limsup_{s\downarrow 0} \frac{d_{C(t,y(t))}(z(t+s))-d_{C(t,y(t))}(z(t))}{s},
\end{aligned}
\end{equation*} 
where $\kappa_r$ is the constant given by \ref{Hxlip} with $r:=\sup_{t\in I} \Vert z(t)\Vert<+\infty$.
\item For all $t\in \operatorname{int}(I)$, where $\dot{z}(t)$ exists,
\begin{equation*}
\limsup_{s\downarrow 0} \frac{d_{C(t,y(t))}(z(t+s))-d_{C(t,y(t))}(z(t))}{s}\leq \max_{y^*\in \partial d(z(t),C(t,y(t)))} \left\langle y^*, \dot{z}(t)\right \rangle.
\end{equation*}
\item For  all $t\in \{ t\in ]0,T[\colon z(t)\notin C(t,y(t))\}$, where $\dot{z}(t)$ exists,
\begin{equation*}
\lim_{s\downarrow 0} \frac{d_{C(t,y(t))}(z(t+s))-d_{C(t,y(t))}(z(t))}{s}= \min_{y^*\in \partial d(z(t),C(t,y(t)))} \left\langle y^*, \dot{z}(t)\right \rangle.
\end{equation*}
\item\label{measurability} For every $x\in \H$ the set-valued map $t\rightrightarrows  \partial d(\cdot,C(t,y(t)))(\A(x))$ is measurable.
\end{enumerate}
\end{lemma}
\begin{proof} 
(iii) and (iv) follows from \cite[Lemma~4.4]{JV-Moreau}. To prove (i), let us consider $\psi\colon I\to \mathbb{R}$ be the function defined by $$\psi(t):=d(z(t),C(t,y(t))),$$where $z(t):=\A(x(t))$ for all $t\in I$. Fix $r=\sup_{t\in I} \Vert z(t)\Vert<+\infty$. Then, it follows from \ref{Hxlip} the existence of $\kappa_r\geq 0$ and $L\in [0,m[$ such that 
$$
\vert \psi(t+h)-\psi(t)\vert \leq \kappa_r \vert t-s\vert+L\Vert z(t+h)-z(t)\Vert, 
$$ 
which implies the absolute continuity of $t \to d(z(t),C(t,y(t)))$.  \newline
To prove (ii), consider $t\in \operatorname{int}(I)$ where $\dot{y}(t)$ exists. Then, for $s>0$ small enough,
\begin{equation*}
\begin{aligned}
\frac{\psi(t+s)-\psi(t)}{s}&=\frac{d(z(t+s),C(t+s,y(t+s)))-d(z(t+s),C(t,y(t)))}{s}\\
&+\frac{d(z(t+s),C(t,y(t)))-d(z(t),C(t,y(t)))}{s}\\
&\leq \kappa_r+ {L\frac{\Vert y(t+s)-y(t)\Vert}{s}}\\
&+\frac{d(x(t+s),C(t,z(t)))-d(x(t),C(t,z(t)))}{s},
\end{aligned}
\end{equation*} 
and taking the superior limit, we get the desired inequality.
\qed
\end{proof}

The following result shows that the set-valued map $(t,x)\rightrightarrows \frac{1}{2}\partial d^2_{C(t,x)}(\A(x))$ satisfies the conditions of Theorem \ref{bothe-2}.
\begin{proposition}\label{distance-prop-2}
  Assume that \ref{HA2}, \ref{Hxlip}, \ref{Hxsubsmooth} and \ref{Hxcomp} hold. Then, the set-valued map $G\colon I\times \H\rightrightarrows \H$ defined by $G(t,x):=\frac{1}{2}\partial d^2_{C(t,x)}(\A(x))$ satisfies:
\begin{enumerate}[label=(\roman{*})]
  \item for all $x\in \H$ and all $t\in I$, $G(t,x)=\A(x)-\operatorname{cl}{\operatorname{co}}\operatorname{Proj}_{C(t,x)}(\A(x))$.
  \item for every $x\in \H$ the set-valued map $G(\cdot,x)$ is measurable.
  \item for every $t\in I$, $G(t,\cdot)$ is upper semicontinuous from $\H$ into $\H_w$.
  \item for every $t\in I$ and $B\subset \H$ bounded, $\gamma\left(G(t,B)\right)\leq (M+k(t))\gamma\left(B\right)$, where $\gamma=\alpha$ or $\gamma=\beta$ is the Kuratowski or the Hausdorff measure of non-compactness of $B$ and $k\in L^1(0,T)$ is given by \ref{Hxcomp}.
\item Let $\A(x_0)\in C(0,x_0)$. Then, for all $t\in I$ and $x\in \H$,
\begin{equation*}
  \Vert G(t,x)\Vert:=\sup\left\{\Vert w\Vert \colon w\in G(t,x)\right\} \leq (M+L)\Vert x-x_0\Vert +\kappa_r\left|t\right|,
\end{equation*}
where $r=\max\{\Vert \A (x)\Vert, \Vert \A(x_0)\Vert\}$ and $\kappa_r$ is the constant given by \ref{Hxlip}.
\end{enumerate}
\end{proposition}
\begin{proof} (i), (ii) and (iii) follow, respectively, from Lemma \ref{lemma-formula}, \ref{measurability} of Lemma \ref{derivadas} and Lemma \ref{usc}. To prove (iv), let $B\subset \H$ be a bounded set included in the ball $r\mathbb{B}$, for some $r>0$. Define the set-valued map  $F(t,x):=\operatorname{Proj}_{C(t,x)}(\A(x))$. Then, for every $t\in I$, 
    $
    \Vert F(t,B)\Vert:=\sup\{\Vert w\Vert \colon  w\in F(t,B)\}\leq \tilde{r}(t),
     $ 
     where $\tilde{r}_B(t):= \kappa_{R} t+L r+L\Vert x_0\Vert+\Vert \A(x_0)\Vert+2R$.
    Indeed, let $z\in F(t,B)$, then there exists $x\in B$ such that $z\in \operatorname{Proj}_{C(t,x)}(\A(x))$. Define $$R:=\max\{\sup_{x\in B}\Vert \A(x)\Vert, \Vert \A(x_0)\Vert\}.$$
    Thus,
     \begin{equation*}
     \begin{aligned}
     \Vert z\Vert &\leq d_{C(t,x)}(\A(x))-d_{C(0,x_0)}(\A(x_0))+\Vert \A(x)\Vert\\
     &\leq \kappa_{R}t+L\Vert x-x_0\Vert+\Vert \A(x)-\A(x_0)\Vert+\Vert \A(x)\Vert\\
      &\leq \kappa_{R}t+L r+L\Vert x_0\Vert+\Vert \A(x_0)\Vert+2R=\tilde{r}_B(t),
     \end{aligned}
     \end{equation*}
     where $\kappa_R$ is the constant given by \ref{Hxlip}. Therefore,
    \begin{equation*}
    \begin{aligned}
    \gamma\left(G(t,B)\right)&\leq \gamma(\A(B))+ \gamma\left(\operatorname{cl}{\operatorname{co}}F(t,B)\right) \\
    &\leq M\gamma(B)+\gamma\left(F(t,B)\cap \tilde{r}_B(t)\mathbb{B}\right) \\
    &\leq M\gamma(B)+\gamma\left(C(t,B)\cap \tilde{r}_B(t)\mathbb{B}\right) \\
    &\leq  (1+k(t))M\gamma(B),
    \end{aligned}
  \end{equation*}
  where we have used \ref{HA2} and  the last inequality is due to \ref{Hxcomp}.
  
  To prove (v), define $\tilde{G}(t,x):=\A(x)-\operatorname{Proj}_{C(t,x)}(\A(x))$. Then, due to \ref{Hxlip},
  \begin{equation*}
  \begin{aligned}
      \Vert \tilde{G}(t,x)\Vert&=d(\A(x),C(t,x))-d(\A(x_0),C(0,x_0))\\
      &\leq (M+L)\Vert x-x_0\Vert+\kappa_r  t ,\\
      \end{aligned}
  \end{equation*}
  where $\kappa_r$ is given by \ref{Hxlip} with $r=\max\{\Vert \A (x)\Vert, \Vert \A(x_0)\Vert \}$. Then, 
  by passing to the closed convex hull in the last inequality, we get the result. \qed
\end{proof}
When the sets $C(t,x)$ are independent of $x$, the subsmoothness in Proposition \ref{distance-prop-2} is no longer required. The following result follows in the same way as Proposition \ref{distance-prop-2}.

\begin{proposition}\label{distance-prop}
  Assume that \ref{HA2} \ref{Hlip} and \ref{Hcomp} hold. Then, the set-valued map $G\colon I\times H \rightrightarrows H$ defined by $G(t,x):=\frac{1}{2}\partial d^2_{C(t)}(\A(x))$ satisfies:
\begin{enumerate}[label=(\roman{*})]
  \item for all $x\in \H$ and all $t\in I$, 
  $
  G(t,x)=\A(x)-\operatorname{cl}{\operatorname{co}}\operatorname{Proj}_{C(t)}(\A(x)).
  $
  \item for every $x\in \H$ the set-valued map $G(\cdot,x)$ is measurable.
  \item for every $t\in I$,  $G(t,\cdot)$ is upper semicontinuous from $\H$ into $\H_w$.
  \item for every $t\in I$ and $B\subset \H$ bounded, $\gamma\left(G(t,B)\right)\leq M\gamma\left(B\right)$, where $\gamma=\alpha$ or $\gamma=\beta$ is the Kuratowski or the Hausdorff measure of non-compactness of $B$ and $k\in L^1(0,T)$ is given by \ref{Hxcomp}.
\item Let $\A(x_0)\in C(0)$. Then, for all $t\in I$ and $x\in \H$,
\begin{equation*}
  \Vert G(t,x)\Vert:=\sup\left\{\Vert w\Vert \colon w\in G(t,x)\right\} \leq M\Vert x-x_0\Vert +\kappa_r t,
\end{equation*}
where $r=\max\{\Vert \A (x)\Vert, \Vert \A(x_0)\Vert \}$ and $\kappa_r$ is the constant given by \ref{Hxlip}.
\end{enumerate}
\end{proposition}

\section{Existence Results for Degenerate State-Dependent Sweeping Processes}\label{Existence}
In this section, we prove the existence of solutions Lipschitz solutions of the degenerate state-dependent sweeping process
\begin{equation}\label{Sweeping-process}\tag{$\mathcal{SP}$}
\left\{
\begin{aligned}
-\dot{x}(t)&\in N_{C(t,x(t))}\left(\A(x(t))\right) & \textrm{ a.e. } t\in I,\\
x(0)&=x_0,
\end{aligned}
\right.
\end{equation}
where  $\A(x_0)\in C(0,x_0)$ and $I=[0,T]$ for some $T>0$.

\noindent Let $\lambda>0$ and consider the following differential inclusion
\begin{equation}\label{Plambda}\tag{$\mathcal{P}_{\lambda}$}
  \left\{
  \begin{aligned}
    -\dot{x}_{\lambda}(t)&\in \frac{1}{2\lambda}\partial d_{C(t,x_{\lambda}(t))}^2\left(\A(x_{\lambda}(t))\right) & \textrm{ a.e. } t\in I,\\
    x_{\lambda}(0)&=x_0,
  \end{aligned}
      \right.
\end{equation}
where $\A(x_0)\in C(0,x_0)$. The following proposition follows from Theorem \ref{bothe-2} and Proposition \ref{distance-prop-2}.
\begin{proposition}\label{Existence-state}
 Assume that \ref{HA2}, \ref{Hxlip}, \ref{Hxsubsmooth} and \ref{Hxcomp} hold. Then, for every $\lambda>0$ there exists at least one absolutely continuous solution $x_{\lambda}$ of (\ref{Plambda}).
\end{proposition}
Let us define $\varphi_{\lambda}(t):=d_{C(t,x_{\lambda}(t))}(\A(x_{\lambda}(t)))$ for $t\in I$.
\begin{remark}\label{def-rho}
Recall that under  \ref{Hxsubsmooth}, according to Proposition \ref{alpha-far},  for every $\alpha\in ]\sqrt{\frac{L}{m}},1]$ there exists $\rho>0$ such that \ref{Hxalpha} holds.
\end{remark}

The following proposition establishes that trajectories of (\ref{Plambda}) stay uniformly close (with respect to $\lambda$) to the moving sets in a small interval of $I=[0,T]$.
\begin{proposition}\label{prop-cotas}
Assume, in addition to the hypotheses of Proposition \ref{Existence-state}, that \ref{HA1} holds. Then, for every $R>0$ there exists $\tau_R\in ]0,T]$ (independent of $\lambda$)  such that if $\lambda<\frac{(m\alpha^2-L) }{\tilde{\kappa}} \rho$,
\begin{equation}\label{cota-phi-1}
\begin{aligned}
\dot{\varphi}_{\lambda}(t)&\leq \tilde{\kappa}+\frac{L-m\alpha^2}{\lambda}\varphi_{\lambda}(t) & \textrm{ a.e. } t\in [0,\tau_R],
\end{aligned}
\end{equation}
where $\alpha\in ]\sqrt{\frac{L}{m}},1]$ and $\rho>0$ are given by Remark \ref{def-rho} and $\tilde{\kappa}:=\kappa_{\Vert \A(x_0)\Vert +MR}$ is the constant given by \ref{Hxlip}. 
Moreover,
\begin{equation}\label{cota-phi-2}
\begin{aligned}
\varphi_{\lambda}(t)&\leq \frac{\tilde{\kappa} \lambda}{m\alpha^2-L} \textrm{ for all }  t\in [0,\tau_R].
\end{aligned}
\end{equation}
\end{proposition}
\begin{proof} Fix $R>0$  and define the set
$$
\Omega_{\lambda}=\{t\in I\colon \Vert x_{\lambda}(t)-x_0\Vert >R\}.
$$
On the one hand,  if $\Omega_{\lambda}=\emptyset$, then $x_{\lambda}(t)\in \mathbb{B}(x_0,R)$ for all $t\in I$. On the one hand, if $\Omega_{\lambda}\neq\emptyset$, then we can define
\begin{equation*}
\tau_{\lambda}:=\inf \{t\in I\colon \Vert x_{\lambda}(t)-x_0\Vert >R\}>0. 
\end{equation*}
Therefore, in what follows we can assume that 
\begin{equation}\label{xbounded}
 x_{\lambda}(t)\in  \mathbb{B}(x_0,R) \textrm{ for all } t\in [0,\tau_{\lambda}],
\end{equation} 
where we have  set   $\tau_{\lambda}=T$ if $\Omega_{\lambda}=\emptyset$. Thus, by virtue of \ref{HA2} and \eqref{xbounded}, it follows that 
\begin{equation}\label{Abounded}
\Vert \A(x_{\lambda}(t))\Vert \leq \Vert \A(x_0)\Vert +MR \textrm{ for all } t\in [0,\tau_{\lambda}]. 
\end{equation}
 According to Proposition \ref{Existence-state}, the function $x_{\lambda}$ is absolutely continuous. Moreover, due to \ref{Hxlip} and \eqref{Abounded},  for every $t,s\in I$
\begin{equation*}
  \left|\varphi_{\lambda}(t)-\varphi_{\lambda}(s)\right|\leq (M+L)\Vert x_{\lambda}(t)-x_{\lambda}(s)\Vert +\tilde{\kappa} |t-s|,
\end{equation*}
where $\tilde{\kappa}:=\kappa_{\Vert \A(x_0)\Vert +MR}$ is the constant given by \ref{Hxlip}. Hence $\varphi_{\lambda}$ is absolute continuous. On the one hand, let $t\in  [0,\tau_{\lambda}]$ where $\varphi_{\lambda}(t)\in ]0,\rho[$ and $\dot{x}_{\lambda}(t)$ exists (recall that $\varphi_{\lambda}(0)=0$). Then, by using (iii) from Lemma \ref{derivadas}, we have
\begin{equation*}
  \begin{aligned}
    \dot{\varphi}_{\lambda}(t)&\leq \tilde{\kappa} +L\Vert \dot{x}_{\lambda}(t)\Vert+\min_{w\in \partial d_{C(t,x_{\lambda}(t))}(z_{\lambda}(t))}\left\langle w,\dot{z}_{\lambda}(t)\right\rangle\\
    &\leq \tilde{\kappa}+\frac{L}{\lambda}\varphi_{\lambda}(t)-\frac{m\alpha^2}{\lambda}\varphi_{\lambda}(t)\\
    &= \tilde{\kappa} - \frac{m\alpha^2-L}{\lambda}\varphi_{\lambda}(t),
  \end{aligned}
\end{equation*}
where we have used \ref{Hxsubsmooth} and Proposition \ref{alpha-far}.

On the other hand,  let $t\in \varphi_{\lambda}^{-1}\left(\{0\}\right)\cap [0,\tau_{\lambda}]$ where $\dot{x}_{\lambda}(t)$ exists. Then, according to ($\mathcal{P}_{\lambda}$), $\Vert \dot{x}_{\lambda}(t)\Vert =0$. Indeed,
\begin{equation*}
\begin{aligned}
\Vert \dot{x}_{\lambda}(t)\Vert \leq \frac{1}{2\lambda}\sup\{\Vert z\Vert \colon z\in \partial d_{C(t,x_{\lambda}(t))}^2\left(\A(x_{\lambda}(t))\right)\} \leq \frac{\varphi_{\lambda}(t)}{\lambda}=0,
\end{aligned}
\end{equation*}
where we have used the identity $\partial d_S^2(x)=2d_S(x)\partial d_S(x)$. Then, due to \ref{Hxlip},  
\begin{equation*}
  \begin{aligned}
    \dot{\varphi}_{\lambda}(t)&=\lim_{h\downarrow 0}\frac{1}{h}\left(d_{C(t+h,x_{\lambda}(t+h))}(z_{\lambda}(t+h))-d_{C(t,x_{\lambda}(t))}(z_{\lambda}(t+h))\right)\\
    &+d_{C(t,x_{\lambda}(t))}(z_{\lambda}(t+h))\\
        &\leq \tilde{\kappa}+L\Vert \dot{x}_{\lambda}(t)\Vert+\lim_{h\downarrow 0}\frac{1}{h}d_{C(t,x_{\lambda}(t))}(z_{\lambda}(t+h))\\
    &\leq \tilde{\kappa} +(M+L)\Vert \dot{z}_{\lambda}(t)\Vert\\
    &\leq  \tilde{\kappa} +\frac{M+L}{\lambda}\varphi_{\lambda}(t)\\
    &=\tilde{\kappa} -\frac{m\alpha^2-L}{\lambda}\varphi_{\lambda}(t).
  \end{aligned}
\end{equation*}
Also, we have that $\varphi_{\lambda}(t)<\rho$ for all $t\in [0,\tau_{\lambda}]$. Otherwise, since $\varphi_{\lambda}^{-1}\left(]-\infty,\rho[\right)$ is open and  $0\in \varphi_{\lambda}^{-1}\left(]-\infty,\rho[\right)$, there would exist $t^*\in ]0,\tau_{\lambda}]$ such that $[0,t^*[\subset  \varphi_{\lambda}^{-1}\left(]-\infty,\rho[\right)$ and $\varphi_{\lambda}(t^*)=\rho$. Then,
\begin{equation*}
\begin{aligned}
  \dot{\varphi}_{\lambda}(t)&\leq \tilde{\kappa} -\frac{m\alpha^2-L}{\lambda}\varphi_{\lambda}(t) & \textrm{ a.e. } t\in [0,t^*[,
  \end{aligned}
\end{equation*}
which, by virtue of Gr\"{o}nwall's inequality, entrain that, for every $t\in [0,t^*[$
\begin{equation*}
\begin{aligned}
  \varphi_{\lambda}(t)\leq \frac{\tilde{\kappa} \lambda}{m\alpha^2-L}\left(1-\exp\left(-\frac{m\alpha^2-L}{\lambda}t\right)\right)\leq \frac{\tilde{\kappa} \lambda}{m\alpha^2-L}< \rho,
\end{aligned}
\end{equation*}
that implies that $\varphi_{\lambda}(t^*)<\rho$, which is not possible. Thus, we have proved that $\varphi_{\lambda}$ satisfies (\ref{cota-phi-1}) and (\ref{cota-phi-2}) in the interval $[0,\tau_{\lambda}]$. \newline
\noindent  It remains to prove that there exists $\tau_R\in ]0,T]$ such that 
\begin{equation}\label{tauR}
\tau_R \leq \tau_{\lambda} \textrm{ for all } \lambda <\frac{(m\alpha^2-L)}{\tilde{\kappa}}\rho.
\end{equation}
Indeed, since $x_{\lambda}$ satisfies $(\mathcal{P}_{\lambda})$, we have
\begin{equation*}
\Vert \dot{x}_{\lambda}(t)\Vert \leq \frac{1}{2\lambda}\sup\{\Vert z\Vert \colon z\in \partial d_{C(t,x_{\lambda}(t))}^2\left(\A(x_{\lambda}(t))\right)\} \leq \frac{\varphi_{\lambda}(t)}{\lambda}\leq \dfrac{\tilde{\kappa}}{m\alpha^2-L},
\end{equation*}
where we have used the identity $\partial d^2_S(x)=2d_S(x)\partial d_S(x)$.  Thus, $x_{\lambda}$ is $\frac{\tilde{\kappa}}{m\alpha^2-L}$ Lipschitz on $[0,\tau_{\lambda}]$. Hence, by the continuity of $x_{\lambda}$ and the definition of $\tau_{\lambda}$, we obtain
\begin{equation*}
\begin{aligned}
R\leq \Vert x_{\lambda}(\tau_{\lambda})-x_0\Vert \leq \int_{0}^{\tau_{\lambda}}\Vert \dot{x}_{\lambda}(s)\Vert ds\leq \frac{\tilde{\kappa}}{m\alpha^2-L}\, \tau_{\lambda},
\end{aligned}
\end{equation*}
which implies the existence of $\tau_R\in ]0,T]$ such that \eqref{tauR} holds. Finally, we have prove that (\ref{cota-phi-1}) and (\ref{cota-phi-2}) holds for the interval $[0,\tau_R]$, which ends the proof. 
\qed
\end{proof}
As a consequence of the last proposition, we obtain that $x_{\lambda}$ is $\frac{\tilde{\kappa}}{m\alpha^2-L}$-Lipschitz on $[0,\tau_R]$.
\begin{corollary} 
  For every $\lambda<\frac{m\alpha^2-L}{\tilde{\kappa}}\rho$ the function $x_{\lambda}$ is $\frac{\tilde{\kappa}}{m\alpha^2-L}$-Lipschitz on $[0,\tau_R]$.
\end{corollary}

Let $(\lambda_n)_n$ be a sequence converging to $0$. The next result shows the existence of a subsequence $\left(\lambda_{n_k}\right)_{k}$ of $(\lambda_n)_n$ such that $\left(x_{\lambda_{n_k}}\right)_k$ converges (in the sense of Lemma \ref{compactness}) to a  solution  of (\ref{Sweeping-process}) over all the interval $I$. We perform the analysis in the interval $[0,\tau_R]$ for some $R$ and then we extend iteratively the solution to all the interval.   
The next theorem extends all known results on State-Dependent Sweeping Processes and Degenerate Sweeping Processes.
\begin{theorem}\label{Lipschitz-state}
Assume that \ref{HA1}, \ref{HA2}, \ref{Hxlip}, \ref{Hxsubsmooth} and \ref{Hxcomp} hold. Then, there exists at least one solution $x\in \operatorname{AC}\left(I;\H\right)$ of (\ref{Sweeping-process}). 
\end{theorem}
\begin{proof} Fix $R>0$ and $\tau_R$ from Proposition \ref{prop-cotas}. Then, 
 by virtue of Proposition \ref{prop-cotas}, the sequences $(x_{\lambda_n})_n$ and $(\A(x_{\lambda_n}))_n$ satisfy the hypotheses of Lemma \ref{compactness}  over the interval $[0,\tau_R]$ with $\psi(t):=\frac{\tilde{\kappa}}{m\alpha^2-L}$ and $\tilde{\psi}(t):=\frac{\tilde{\kappa}M}{m\alpha^2-L}$, respectively. Therefore, there exist subsequences $(x_{\lambda_{n_k}})_k$ and  $(\A(x_{\lambda_{n_k}}))_k$ of $(x_{\lambda_n})_n$  and $(\A(x_{\lambda_n}))_n$, respectively and functions $x\colon [0,\tau_R]\to \H$ and $\A x\colon [0,\tau_R]\to \H$ satisfying the hypotheses \ref{comp-i}-\ref{comp-iv} of Lemma \ref{compactness}. For simplicity, we write $x_k$ and $A(x_k)$ instead of $x_{\lambda_{n_k}}$ and  $\A(x_{\lambda_{n_k}})$, respectively.

\begin{claim}{1}{ $\left(\A(x_{k}(t))\right)_k$ and $\left(x_{k}(t)\right)_k$  are relatively compact in $\H$ for all $t\in [0,\tau_R]$.}
\end{claim}
\begin{claimproof}{1}{Let $t\in [0,\tau_R]$. Let us consider $y_{k}(t)\in \operatorname{Proj}_{C(t,x_{k}(t))}\left(\A(x_{k}(t))\right)$. Then, $\Vert \A(x_{k}(t))-y_{k}(t)\Vert=d_{C(t,x_{k}(t))}\left(\A(x_{k}(t))\right)$. Thus,
\begin{equation*}
\begin{aligned}
  \Vert y_{k}(t)\Vert &\leq d_{C(t,x_{k}(t))}\left(\A(x_{k}(t))\right)+\Vert \A(x_k(t))\Vert\\
  &\leq \frac{\tilde{\kappa} \lambda_{n_k}}{m\alpha^2-L}+M\Vert x_k(t)-x_0\Vert +\Vert \A(x_0)\Vert\\
  &\leq \tilde{r}(t):=\frac{\tilde{\kappa}}{m\alpha^2-L}\left(\lambda_{n_k}+Mt\right)+\Vert \A(x_0)\Vert.
\end{aligned}
\end{equation*}
Also, since $(\A(x_k(t))-y_k(t))$ converges to $0$,
\begin{equation*}
  \gamma\left(\left\{\A(x_{k}(t))\colon k\in \mathbb{N}\right\}\right)=\gamma\left(\left\{y_{k}(t)\colon k\in \mathbb{N}\right\}\right).
\end{equation*}
 Therefore, if $A:=\left\{\A(x_{k}(t))\colon k\in \mathbb{N}\right\}$,
\begin{equation*}
  \gamma\left(A\right)=\gamma(\left\{y_{k}(t)\colon k\in \mathbb{N}\right\}) \leq \gamma\left(C\left(t,{\left\{x_{k}(t)\colon k\in \mathbb{N}\right\}}\right)\cap \tilde{r}(t)\mathbb{B}\right) \leq k(t)  \gamma\left(A\right),
\end{equation*}
where we have used \ref{Hxcomp}. Finally, since $k(t)<1$, we obtain that $\gamma\left(A\right)=0$, which shows that $(\A(x_k(t)))_k$ is relatively compact in $\H$. Finally, $(x_k(t))_k$ is relatively compact in $\H$ by virtue of \ref{HA2}. 
}
\end{claimproof}

\begin{claim}{2}{ $\A(x(t))\in C(t,x(t))$ for all $t\in [0,\tau_R]$.}
\end{claim}
\begin{claimproof}{2}{As a result of Claim 1 and the weak convergence $x_k(t)\rightharpoonup x(t)$ for all $t\in [0,\tau_R]$ (due to \ref{comp-i} of Lemma \ref{compactness}), we obtain the strong convergence of $(x_k(t))_k$ to $x(t)$ for all $t\in [0,\tau_R]$.
Therefore, due to \ref{Hxlip}, 
\begin{equation*}
\begin{aligned}
  d_{C(t,x(t))}(\A(x(t)))&\leq \liminf_{k\to \infty}\left(d_{C(t,x_{k}(t))}\left(\A(x_{k}(t))\right)+(M+L)\Vert x_{k}(t)-x(t)\Vert \right)\\
  &\leq \liminf_{k\to \infty} \left(\frac{\tilde{\kappa} \lambda_{n_k}}{m\alpha^2-L}+(M+L)\Vert x_{k}(t)-x(t)\Vert\right)=0,
  \end{aligned}
\end{equation*}
which proves the claim.
}\end{claimproof}

Now, we prove that ${x}$ is a solution of (\ref{Sweeping-process}) over the interval $[0,\tau_R]$. Define
\begin{equation*}
\tilde{F}(t,x):=\operatorname{cl}{\operatorname{co}}\left(\frac{\tilde{\kappa}}{m\alpha^2-L}\partial d_{C(t,x)}(\A(x))\cup \{0\}\right) \textrm{ for } (t,x)\in [0,\tau_R]\times \H.
\end{equation*}
Then, for a.e. $t\in [0,\tau_R]$
\begin{equation*}
-\dot{x}_k(t)\in \frac{1}{2\lambda}\partial d_{C(t,x_{k}(t))}^2\left(\A(x_{k}(t))\right)\subset \tilde{F}(t,x_{k}(t)),
\end{equation*}
where we have used Proposition \ref{prop-cotas}.
\begin{claim}{3:}{ $\tilde{F}$ has closed and convex values and satisfies:
\begin{enumerate}[label=(\roman{*})]
\item\label{3.1} for each $x\in H$, $\tilde{F}(\cdot,x)$ is measurable;
\item\label{3.2} for all $t\in [0,\tau_R]$, $\tilde{F}(t,\cdot)$ is upper semicontinuous from $\H$ into $\H_w$;
\item\label{3.3} if $\A(x)\in C(t,x)$ then $\tilde{F}(t,x)=\frac{\tilde{\kappa}}{m\alpha^2-L}\partial d_{C(t,x)}(\A(x))$.
\end{enumerate}
}\end{claim}
\begin{claimproof}{3}{ To prove Claim 3, we follow the ideas from \cite[Theorem~6.1]{JV-Moreau}. Define $G(t,x):=\frac{\tilde{\kappa}}{m\alpha^2-L}\partial d_{C(t,x)}(\A(x))\cup\{0\}$. We note that $G(\cdot,x)$ is measurable as the union of two measurable set-valued maps (see \cite[Lemma~18.4]{Aliprantis}). Let us define $\Gamma(t):=\tilde{F}(t,x)$. Then, $\Gamma$ takes weakly compact convex values. Fixing any $d\in \H$, by virtue of \cite[Proposition~2.2.39]{PapaHandbook-1}, is enough to verify that the support function $t\mapsto \sigma(d,\Gamma(t)):=\sup\{\left\langle v,d\right\rangle\colon v\in \Gamma(t)\}$ is measurable. Thus,
\begin{equation*}
\sigma(d,\Gamma(t)):=\sup\{\left\langle v,d\right\rangle\colon v\in \Gamma(t)\}=\sup\{\left\langle v,d\right\rangle \colon v\in G(t,x)\},
\end{equation*}
is measurable because $G(\cdot,x)$ is measurable. Thus \ref{3.1} holds. Assertion \ref{3.2} follows directly from \cite[Theorem~17.27 and 17.3]{Aliprantis}. Finally, if $\A(x)\in C(t,x)$ then $0\in \partial d_{C(t,x)}(\A(x))$. Hence, using the fact that the subdifferential of a locally Lipschitz function is closed and convex,
\begin{equation*}
\tilde{F}(t,x)=\operatorname{cl}{\operatorname{co}}\left(\frac{\tilde{\kappa}}{m\alpha^2-L}\partial d_{C(t,x)}(\A(x))\right)=\frac{\tilde{\kappa}}{m\alpha^2-L}\partial d_{C(t,x)}(\A(x)),
\end{equation*}
which shows \ref{3.3}.
}\end{claimproof}

In summary, we have
\begin{enumerate}[label=(\roman{*})]
\item for each $x\in \H$, $\tilde{F}(\cdot,x)$ is measurable.
\item for all $t\in [0,\tau_R]$, $\tilde{F}(t,\cdot)$ is upper semicontinuous from $\H$ into $\H_w$.
\item $\dot{x}_k \rightharpoonup \dot{x}$ in $L^1\left([0,\tau_R];\H\right)$ as $k\to +\infty$.
\item ${x}_k(t) \to {x}(t)$ as $k\to +\infty$ for all $t\in [0,\tau_R]$.
\item $\A({x}_k(t)) \to \A({x}(t))$ as $k\to +\infty$ for all $t\in [0,\tau_R]$.
\item $-\dot{x}_k(t)\in \tilde{F}\left(t,x_{k}(t)\right)$ for a.e. $t\in [0,\tau_R]$.
\end{enumerate}
These conditions and the Convergence Theorem (see \cite[p.60]{Aubin-Cellina} for more details) imply that $x$  satisfies 
\begin{equation*}
\left\{
\begin{aligned}
-\dot{x}(t)&\in \tilde{F}(t,x(t)) & \textrm{ a.e. } t\in [0,\tau_R],\\
x(0)&=x_0\in C(0,x_0),
\end{aligned}
\right.
\end{equation*}
which, according to Claim 3, implies that $x$ is a solution of
\begin{equation*}
\left\{
\begin{aligned}
-\dot{x}(t)&\in \frac{\tilde{\kappa}}{m\alpha^2-L}\partial d_{C(t,x(t))}(\A(x(t))) & \textrm{ a.e. } t\in [0,\tau_R],\\
x(0)&=x_0\in C(0,x_0).
\end{aligned}
\right.
\end{equation*}
Therefore, by virtue of (\ref{cono-distance}) and Claim 2, $x$ is a solution of (\ref{Sweeping-process}) over $[0,\tau_R]$. Finally, to extend the solution over all the interval $I$, we can use \cite[Lemma~5.7]{JV-alpha}.
\qed
\end{proof}

\subsection{\textbf{The Case of the Degenerate Sweeping Process}} \
This subsection is devoted to the degenerate  sweeping process:
\begin{equation}\label{SP-normal}
\left\{
\begin{aligned}
-\dot{x}(t)&\in N\left(C(t);\A(x(t))\right) & \textrm{ a.e. } t\in I,\\
x(0)&=x_0\in C(0).
\end{aligned}
\right.
\end{equation}
This differential inclusion can be seen as a particular case of (\ref{Sweeping-process}) when the sets $C(t,x)$ are state independent. We show that Theorem \ref{Lipschitz-state} is valid under the weaker hypothesis \ref{Halpha} instead of \ref{Hxsubsmooth}. The following theorem improves the results from \cite{JV-alpha,JV-Galerkin2017} regarding the existence for sweeping process with nonregular sets and also complements the results from \cite{Adly2020,Haddad-Sene2019} regarding the existence of degenerate sweeping process with uniformly prox-regular sets. 
\begin{theorem}
Assume that \ref{HA1}, \ref{HA2}, \ref{Hlip}, \ref{Halpha} and \ref{Hcomp} hold. Then, there exists at least one solution $x\in \operatorname{AC}\left(I;\H\right)$ of (\ref{SP-normal}). 
\end{theorem}
\begin{proof}
  According to the proof of Theorem \ref{Lipschitz-state}, we observe that \ref{Hxsubsmooth} was used to obtain \ref{Hxalpha} and the upper semicontinuity of $\partial d_{C(t,\cdot)}(\cdot)$ from $\H$ into $\H_w$ for all $t\in I$. Since in the present case these two properties hold under \ref{Halpha} (see Proposition \ref{distance-prop}), it is sufficient to adapt the proof of Theorem \ref{Lipschitz-state} to get the result.
\qed
\end{proof}

\section{Conclusions}

In this paper, we have introduced and studied the degenerate state-dependent sweeping process. By a technical modification of the arguments given in \cite{JV-Moreau}, we proved the existence of solutions for Lipschitz moving sets with respect to the truncated Hausdorff distance. We generalize existing results in the literature for both the state-dependent sweeping process and the degenerate sweeping process.  Regarding the degenerate state-dependent sweeping processes, there remain many questions that require further investigations. For example, it would be interesting to study the case of bounded variation degenerate state-dependent sweeping process. In this regard, the work of Recupero \cite{Recupero-2015,Recupero-2015-JCA,Recupero2020} seems promising to extend the present result to the framework of differential measures.  Another unexplored research topic is the optimal control of degenerate state-dependent sweeping processes, which is of particular interest to practitioners. 
We hope to address these challenges in the future.

\begin{acknowledgements}
Diana Narv\'aez was supported by ANID-Chile under grants CMM-AFB170001, Fondecyt Regular N$^{\circ}$ 1171854, and Fondecyt Regular N$^{\circ}$ 1190012 Etapa 2020. 
Emilio Vilches was partially supported by ANID-Chile under grants Fondecyt de Iniciaci\'on N$^{\circ}$ 11180098 and Fondecyt Regular N$^{\circ}$ 1200283.
\end{acknowledgements}

\bibliographystyle{unsrt}
\bibliography{bib}

\end{document}